\documentclass[12pt]{amsart}
\usepackage{amsmath, amscd}
\usepackage{amsfonts}
\usepackage{amssymb}
\usepackage{amsthm}
\usepackage{verbatim}
\allowdisplaybreaks
\usepackage{graphicx}
-\marginparwidth \oddsidemargin -9mm \evensidemargin \oddsidemargin
\advance\hoffset by 5mm

\newcommand{\R}{{\mathbf R}}

\newtheorem{theorem}{Theorem}[section] 
\newtheorem{lemma}[theorem]{Lemma} 
\newtheorem{proposition}[theorem]{Proposition}
\newtheorem{corollary}[theorem]{Corollary}
\theoremstyle{definition} 

\theoremstyle{remark}

\newtheorem*{remark*}{Remark}

\begin{document}


\title{Finite time blow up in the hyperbolic Boussinesq system}
\author{Alexander Kiselev}
\address{\hskip-\parindent
Alexander Kiselev\\
Department of Mathematics\\
Rice University\\
Houston, TX 77005, USA}
\email{kiselev@rice.edu}
\author{Changhui Tan}
\address{\hskip-\parindent
Changhui Tan\\
Department of Mathematics\\
Rice University\\
Houston, TX 77005, USA}
\email{ctan@rice.edu}

\date{\today}

\begin{abstract}
In recent work of Luo and Hou \cite{HouLuo1}, a new scenario for finite time blow up in solutions of 3D Euler equation has been proposed.
The scenario involves a ring of hyperbolic points of the flow located at the boundary of a cylinder. In this paper, we propose a two dimensional
model that we call ``hyperbolic Boussinesq system". This model is designed to provide insight into the hyperbolic point blow up scenario.
The model features an incompressible velocity vector field, a simplified Biot-Savart law, and a simplified term modeling buoyancy.
We prove that finite time blow up happens for a natural class of initial data.
\end{abstract}

\maketitle

\section{Introduction}

The Euler equation of fluid mechanics has been derived in 1755 and appears to be the second PDE ever written. The equation is nonlinear and
nonlocal, which makes analysis challenging. In particular, the question whether solutions corresponding to smooth initial data remain globally regular
remains open in three dimensions. There have been many attempts to resolve this problem either in the regularity direction, or by constructing finite time
blow up examples. We refer to \cite{MB,MP} for history and more details.

Recently, a new scenario for finite time blow up in 3D Euler equation has been proposed by Luo and Hou \cite{HouLuo1} based on extensive numerical simulations.
The scenario is axi-symmetric, and is set in a vertical cylinder $r=1$ with no penetration boundary conditions at the boundary and periodic boundary conditions in $z.$
Angular components of both vorticity, $\omega^\theta,$ and velocity $u^\theta$ obey odd symmetry with respect to $z=0$ plane. The resulting solution forms
rolls which make all points satisfying $r=1$ and $z=0$ hyperbolic points of the flow. It is at these points that very fast growth of vorticity $\omega^\theta$ is observed.

It is well known that the 2D Boussinesq system is essentially identical to the 3D axi-symmetric Euler equation away from the axis $r=0$ (see, e.g. \cite{MB}).
Since in the Hou-Luo scenario, the growth happens at the boundary and away from the axis, we will operate with the 2D Boussinesq system directly.
Recall that the 2D Boussinesq system in vorticity form is given by
\begin{eqnarray}\label{vorteq}
\partial_t \omega +(u \cdot \nabla) \omega = \partial_{x_1} \rho, \,\,\, \omega(x,0)=\omega_0(x),\\
\partial_t \rho +(u \cdot \nabla) \rho = 0, \,\,\, \rho(x,0)=\rho_0(x), \label{deneq} \\
u = \nabla^\perp (-\Delta)^{-1} \omega. \label{bseq}
\end{eqnarray}
We will consider this system in the half-space $x_2 \geq 0,$ and in \eqref{bseq} take Laplacian satisfying Dirichlet boundary conditions on the boundary $x_2=0.$
Such choice corresponds to no penetration boundary condition for $u.$ The initial condition $\omega_0(x)$ is odd in $x_1$ and $\rho_0(x)$ is even in $x_1;$ this symmetry is conserved by evolution.
This set up corresponds to Hou-Luo scenario turned by $\pi/2:$ $x_1$ corresponds to $z$ and $x_2$ to $r,$
and for the right initial data we expect very fast growth of $\omega$ at the origin. We note that, naturally, the problem of global regularity vs finite time blow up for the system
\eqref{vorteq}, \eqref{deneq}, \eqref{bseq} is also open and well known. It is appears, for example, as one of the ``eleven great problems of mathematical hydrodynamics" in \cite{Yud}.

There have been several works which aimed to understand Hou-Luo scenario rigorously. Kiselev and Sverak \cite{KS} have looked at a geometry and initial data similar to Hou-Luo
scenario but in the 2D Euler case, which is equivalent to setting $\rho \equiv 0$ in the 2D Boussinesq system. They constructed examples of solutions in unit disk $D$ for which $\| \nabla \omega\|_{L^\infty}$
exhibits double exponential growth for all times. This is known to be the fastest possible growth rate, as double exponential in time upper bounds on $\| \nabla \omega\|_{L^\infty}$
go back to work of Wolibner \cite{Wol}. A key part of the construction in \cite{KS} is the following representation formula for the fluid velocity near origin, $x_{1,2} \geq 0:$
\begin{equation}\label{ML}
u_{i}(x,t) = (-1)^i \frac{4}{\pi} x_i \int_{Q(x)} \frac{y_1 y_2}{|y|^4} \omega(y,t)\, dy + B_i(x)x_i, \,\,\,i =1,2, \,\,\, |B_i(x)| \leq C\|\omega\|_{L^\infty},
\end{equation}
where $Q(x)$ is the ``look back" set $Q(x) = \{ y \in D: y_1 \geq x_1, \,\,\, y_2 \geq x_2 \}.$ There are also certain small exceptional sectors where \eqref{ML} is not valid, buy we omit these details.
The first term on the right hand side in \eqref{ML} is, in certain regimes, the main term.
It possesses useful features: it is sign definite if $\omega$ has fixed sign, and there is a hidden comparison-like principle based on the increase of $Q$ as $x$ moves closer to the origin.
These properties play an important role in the proof of lower double exponential in time bound on the gradient of vorticity.

Other works focused on 1D models of the Hou-Luo scenario. Hou and Luo \cite{HouLuo1} proposed a model which has one-dimensional structure similar to \eqref{vorteq}, \eqref{deneq}, but with the effective
Biot-Savart law given by $u_x = H\omega,$ where $H$ is the Hilbert transform. This model can be viewed as 2D Boussinesq system restricted to the boundary $x_2=0,$ with the Biot-Savart law
obtained under assumption that vorticity is concentrated in a boundary layer and does not depend on $x_2$ in this boundary layer. A simpler 1D model with Biot-Savart law inspired by \eqref{ML}
has been considered in \cite{CKY}, where it was also proved that finite time blow up can happen for this model. In \cite{HouLiu}, more information on the structure of blow up solutions has been obtained.
Existence of finite time blow up in the original Hou-Luo model has been proved in \cite{CHKLSY}, and a more general argument applying to a broader class of models was presented in \cite{DKX}.
Some infinite energy solutions of 2D Boussinesq system with simple structure and growing derivatives, inspired by Hou-Luo scenario, have been presented in \cite{CCW}.

Passing from the nonlinear analysis of 2D Euler equation \cite{KS} or 1D models \cite{CKY,CHKLSY} to the 2D Boussinesq case presents many challenges. First, one needs to understand how
growth of vorticity happens in 2D geometry, and to develop a framework for controlling it. Secondly, as opposed to the 2D Euler case, the vorticity no longer has fixed sign (in $x_1 \geq 0$ region),
since the forcing term $\partial_{x_1}\rho$ will generate vorticity of the opposite sign. This may deplete flow towards the origin which increases $\partial_{x_1}\rho$ and drives vorticity growth.
Thirdly, analysis of Biot-Savart law that leads to \eqref{ML} fails if vorticity can grow: the terms that go into the Lipschitz error in \eqref{ML} can no longer be controlled the same way.
Each of these complications is significant.

Our goal in this work is to address the first issue, and to develop a fully two dimensional, incompressible model which exhibits finite time blow up. In this process, we will be able to
get some idea of the picture of blow up as well as introduce some relevant objects. The model will have simplified Biot-Savart law and also simplified forcing term.
Similarly to \eqref{vorteq}, \eqref{deneq}, \eqref{bseq} it can also be set on half space, but due to symmetry it suffices to consider the first quadrant $D := \{x \left| x_1 \geq 0, \,\,\,x_2 \geq 0 \right.\}.$
The model is given by
\begin{eqnarray}\label{vortmod}
\partial_t \omega + (u \cdot \nabla) \omega = \frac{\rho}{x_1}, \,\,\, \omega(x,0)=\omega_0(x), \\
\partial_t \rho + (u \cdot \nabla) \rho = 0,  \,\,\, \rho(x,0)=\rho_0(x), \label{denmod} \\
u = (-x_1 \Omega(\eta,t), x_2 \Omega(\eta,t)), \,\,\, \Omega(\eta) = \int_{y_1y_2 \geq \eta}\frac{1}{|y|^2}\omega(y,t)\,dy,\,\,\,\eta=x_1x_2. \label{bsmod}
\end{eqnarray}
Comparing this system with the 2D Boussinesq, note that we replaced $\partial_{x_1} \rho$ with $\frac{\rho}{x_1}.$ Given that we expect blow up to happen at the origin,
and that $\rho$ will initially be supported away from the origin, the second term is a natural model for the first one, and has been proposed in \cite{HR}.
The term $\frac{\rho}{x_1}$
is also simpler, as it is sign definite if $\rho$ has fixed sign. Next, the form of the Biot-Savart law
\begin{equation}\label{bsstruct} u = (-x_1 \Omega(x,t), x_2 \Omega(x,t)) \end{equation}
is patterned after the expression
\eqref{ML}. If one also requires $u$ to be incompressible, then a simple computation shows that $\Omega$ can only depend on $x_1x_2.$ Since every trajectory corresponding to
\eqref{bsmod} is a hyperbola, we see that $\Omega(x_1x_2,t)$ is constant along each trajectory, at any given time. The form of the integral in \eqref{bsmod} defining $\Omega$ is the simplest
possible with the same dimensional structure as the real Biot-Savart law. For a more complete resemblance with \eqref{ML}, we could have taken the kernel in the integral defining $\Omega$
 in \eqref{bsmod} to be $\frac{y_1y_2}{|y|^4},$ but we will indicate below that this change makes no difference in terms of the key properties of the model.
 We will call the system \eqref{vortmod}, \eqref{denmod}, \eqref{bsmod} ``hyperbolic Boussinesq system",
since this model is geared towards the hyperbolic point growth scenario, and the trajectories of the system are precise hyperbolas.

Our main goal in this paper is to prove local well-posedness as well as finite time blow up for hyperbolic Boussinesq system. We will say that $f \in K_n$ if $f$ has compact support
in $D,$ $f \in C^n ( D)$, and $\delta_f = {\rm min}_{x \in {\rm supp}(f)} x_1 >0.$ We set
\[ \|f\|_{K_n} := \|f\|_{C^n(D)}  + |{\rm supp}(f)| + \delta_f^{-1}. \]

\begin{theorem}\label{local}
Suppose $\omega_0,\rho_0 \in K_n,$ $n \geq 1.$ Then there exists $T = T(\|\omega_0\|_{K_n},\|\rho_0\|_{K_n})$ such that there exists a unique solution $\omega(x,t),\,\,\,\rho(x,t)$ of the hyperbolic Boussinesq system
\eqref{vortmod}, \eqref{denmod}, \eqref{bsmod} which belongs to $C([0,T], K_n)$.
\end{theorem}

\begin{theorem}\label{bu}
There exist smooth initial data $\omega_0, \rho_0$ which are in $K_n$ for all $n$ such that the corresponding solution $\omega(x,t), \,\,\,\rho(x,t)$ blows up in finite time.
Specifically, finite time blow up holds in the sense that $\Phi(0,t) \equiv 2\int_0^t \Omega(0,s)\,ds$ as well as
$\int_0^t \|\omega(\cdot, s)\|_{L^\infty}\,ds$ tend to infinity as $t$ approaches the blow up time $T_b.$
\end{theorem}

We note that in a recent independent work \cite{HORY}, a 2D model of Boussinesq system with a different Biot-Savart law has been considered, and finite time blow up has been proved by a very different method
involving lower and upper bounds on the solution. The Biot-Savart law of \cite{HORY} is also given by \eqref{bsstruct}. The difference is in the factor $\Omega$ which is similar to the
integral appearing in the main term of \eqref{ML} with integration restricted to a certain sector for technical reasons. However, such Biot-Savart law does not lead to
incompressible flow.

Also, in a recent preprint \cite{TT}, several new models are proposed for the 3D Euler equation, and finite time blow up is proved for these models. The focus of \cite{TT} is on models that
share as many conservation law properties with 3D Euler equation as possible. The modified Biot-Savart laws of \cite{TT} involve replacement of inverse Laplacian with a multi-scale operator.
Our goal here, on the other hand, is to study the specific blow up scenario and to study models designed to develop intuition as well as framework for its further analysis.

The paper is organized as follows. In Section~\ref{prelims} we introduce some useful explicit formulas for the solution as well as sketch a proof of local well-posedness.
In Section~\ref{2deul} we take a detour and consider the hyperbolic analog of the 2D Euler equation, by setting $\rho \equiv 0$ in the hyperbolic Boussinesq system.
We discover that, in some sense, the 2D hyperbolic Euler is ``less singular" than the true 2D Euler, as its solution satisfies just single in time exponential upper bound on the
derivatives of vorticity (which is qualitatively sharp). We note that the factor $\Omega(x_1x_2,t)$ certainly does not satisfy the bound 
\[ \|\Omega(\cdot,t)\|_{L^\infty} \leq C\|\omega(\cdot,t)\|_{L^\infty},\] 
from which the exponential bound on derivatives would easily follow. Instead, the only bound available is similar
to the 2D Euler case and involves a logarithm of higher order norm such as $\|\nabla \omega\|_{L^\infty}$ or $\|\omega\|_{C^\alpha}.$
In the 2D Euler case, this leads to double exponential growth examples, but the 2D hyperbolic Euler provides an interesting
example where such fast growth does not happen due dynamical depletion of nonlinearity.
Finally, in Section~\ref{blowup} we provide a proof of finite time blow up in solutions of the hyperbolic Boussinesq system.

\section{Preliminaries}\label{prelims}

Our first goal is to establish local well-posedness of the hyperbolic Boussinesq system.
It will be convenient for us to make a change of coordinates $z_1 = \log (x_1 x_2),$ $z_2 = \log (x_2/x_1);$ so $x_1 = e^{\frac{z_1-z_2}{2}},$ $x_2 = e^{\frac{z_1+z_2}{2}}.$
We denote $\tilde{\rho}(z,t)= \rho(x(z),t)$ and $\tilde{\omega}(z,t) = \omega(x(z),t).$ We also define
\begin{equation}\label{tildeOm} \tilde{\Omega}(z_1,t) = \Omega(\eta(z),t) = \frac14 \int_{z_1}^\infty d\tilde{z}_1 \int_{\R} \frac{\tilde{\omega}(\tilde{z},t)}{\cosh \tilde{z}_2}\,d\tilde{z}_2; \end{equation}
the last equality can be verified by a straightforward computation making a coordinate change in the integral for $\Omega$ in \eqref{bsmod}.
The equations for $\tilde{\omega},$ $\tilde{\rho}$ then read
\begin{eqnarray}\label{tildeomega}
\partial_t \tilde{\omega} + 2 \tilde{\Omega} \partial_{z_2} \tilde{\omega} = e^{\frac{z_2-z_1}{2}} \tilde{\rho}, \,\,\, \tilde{\omega}(z,0)= \tilde{\omega}_0(z), \\
\label{tilderho}
\partial_t \tilde{\rho} + 2 \tilde{\Omega} \partial_{z_2} \tilde{\rho} =0, \,\,\, \tilde{\rho}(z,0) = \tilde{\rho}_0(z).
\end{eqnarray}
In the $x$ coordinates, we think of the initial data $\omega_0,\rho_0$ as smooth, non-negative, with support contained in some rectangle $0<\delta \leq x_1 \leq C,$ $0 \leq x_2 \leq C.$ We will see that for the
finite time blow up argument, only $\rho_0$ is important; we can set $\omega_0=0.$ For the finite time blow up argument, we will also need to assume that $\rho_0$ is
not identically zero on the $x_1$ axis. In the $z$ coordinates, this corresponds to $\tilde{\omega}_0,$ $\tilde{\rho}_0$ supported in the half-strip
\begin{equation}\label{zsupp}
-K \leq 2 \log \delta \leq z_1-z_2 \leq 2\log C \leq K, \,\,\, -\infty < z_1+z_2 \leq 2 \log C \leq K,
\end{equation}
where $K$ is some fixed constant that only depends on $\omega_0.$
Moreover, for all $z_1$ small enough, we have $\int_{\R} \tilde{\rho}_0(z_1, z_2)\,dz_2 \geq c>0;$ this follows from continuity of $\rho_0$ and the fact that it does not vanish on the $x_1$ axis.
The structure of the initial data in both systems of coordinates is illustrated on Figure~\ref{fig_init}.

\begin{figure}[htbp]
\begin{center}
\includegraphics[scale=1]{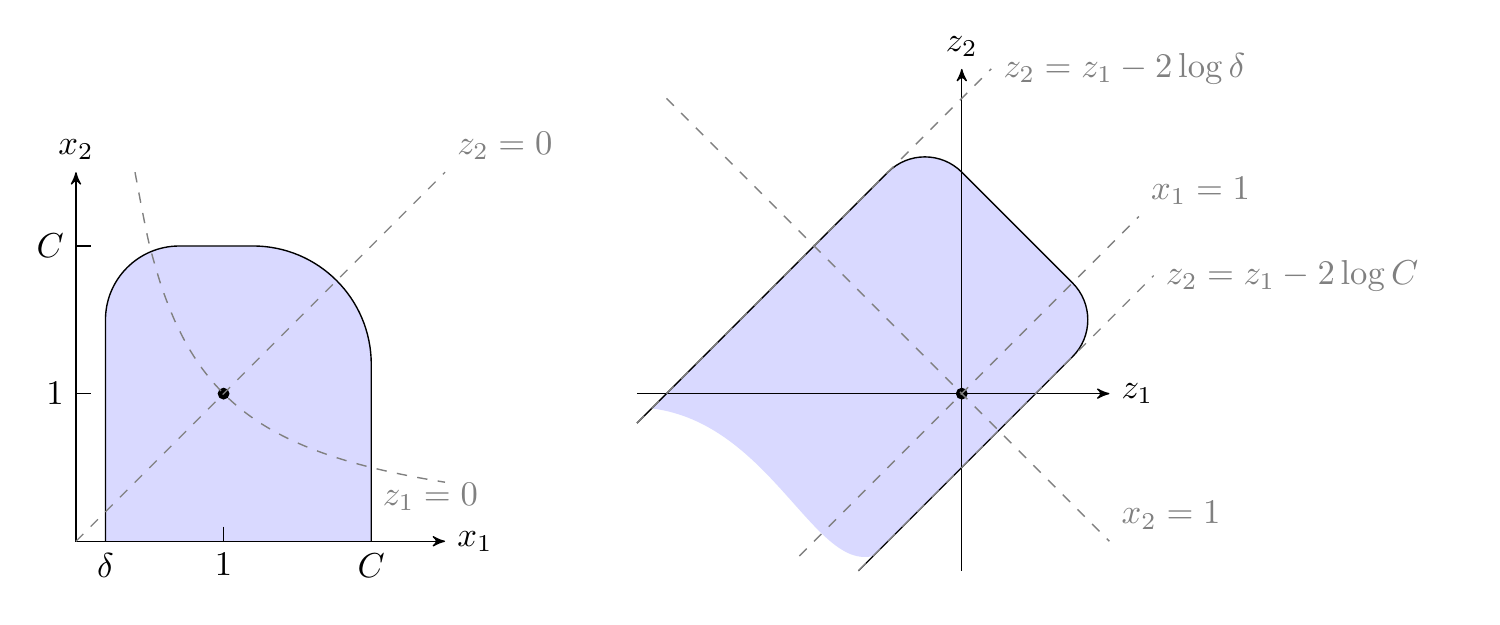}
\caption{The initial data \label{fig_init}}
\end{center}
\end{figure}

For much of the rest of this paper, we will work in the $z$ coordinate representation of the hyperbolic Boussinesq system. Therefore, for the sake of simplicity, we will abuse the notation
and omit $\tilde{ }$ over $\omega,$ $\rho$ and $\Omega$. It will be clear from the context whether we are thinking of these functions in the $z$ or in the original $x$ coordinates.
We can use the method of characteristics to rewrite the system \eqref{tildeomega}, \eqref{tilderho}, \eqref{tildeOm} in an equivalent integral form. Notice that conveniently, in the $z$ coordinates, the first component
of the characteristic does not change. Thus all characteristics are straight lines parallel to $z_2$ axis, and the speed of motion along these lines is modulated by the nonlocal function $\Omega.$
Let us introduce a short cut notation
\begin{equation}\label{phidef}
\Phi(z_1,t) = 2\int_0^t \Omega(z_1,s)\,ds = \frac12 \int_0^t \int_{z_1}^\infty d\tilde{z}_1 \int_{\R} \frac{\tilde{\omega}(\tilde{z}_1,z_2,s)}{\cosh z_2}\,dz_2 ds.
\end{equation}
We obtain
\begin{equation}\label{rhosol}
\rho(z,t) = \rho_0\left(z_1, z_2 - \Phi(z_1,t)\right)
\end{equation}
and
\[ 
\omega(z,t) = \omega_0 \left(z_1, z_2 -\Phi(z_1,t)\right) + \int_0^t f \left(z_1,z_2 - \Phi(z_1,t)+\Phi(z_1,s)\right)\,ds,
\] 
where $f(z,s) = e^{\frac{z_2-z_1}{2}} \rho(z,s).$ Using \eqref{rhosol}, we can rewrite the solution as
\begin{equation}\label{omsol}
\omega(z,t) = \omega_0\left(z_1, z_2 - \Phi(z_1,t)\right)+ f_1\left(z_1, z_2 - \Phi(z_1,t)\right) \int_0^t e^{\frac12 \Phi(z_1,s)}\,ds,
\end{equation}
with $f_1(z_1,z_2) = e^{\frac{z_2-z_1}{2}} \rho_0(z_1,z_2).$ Note that since $\rho_0(z)$ is supported on the finite strip $2 \log \delta \leq z_1-z_2 \leq 2\log C,$ we have that $f_1(z_1, z_2)$
is a bounded function with the same regularity as $\rho_0.$

Let us also provide integral formulas for the solution in the $x$ coordinate representation. These can be obtained directly by solving \eqref{vortmod}, \eqref{denmod}, or by
making a change of coordinates in \eqref{rhosol}, \eqref{omsol}:
\begin{eqnarray}\label{rhox}
\rho(x,t) = \rho_0\left(x_1 e^{\frac12 \Phi(\log(x_1x_2),t)}, 
x_2 e^{-\frac12 \Phi(\log(x_1x_2),t)}\right); \\ 
\omega(x,t) = \omega_0\left(x_1 e^{\frac12 \Phi(\log(x_1x_2),t)}, 
x_2 e^{-\frac12 \Phi(\log(x_1x_2),t)}\right)+ \nonumber \\ 
x_1^{-1} \rho_0\left(x_1 e^{\frac12 \Phi(\log(x_1x_2),t)}, 
x_2 e^{-\frac12 \Phi(\log(x_1x_2),t)} \right) 
\int_0^t e^{-\frac12 (\Phi(\log(x_1x_2), t)-\Phi(\log(x_1x_2),s))}\,ds; \label{vortx}  \\
 \Phi(\log(x_1x_2), t) = \frac14 \int_0^t  \int_{\log(x_1x_2)}^\infty  \int_{\R} \frac{\tilde{\omega}(\tilde{z},s)}{\cosh \tilde{z}_2}\,d\tilde{z}_2 d\tilde{z}_1 ds=
 \label{phieq} \int_0^t \int_{y_1y_2 \geq x_1x_2}\frac{\omega(y,s)}{|y|^2}\,dyds.
\end{eqnarray}

We now begin to discuss the local well-posedness of the hyperbolic Boussinesq. As the first step, let us obtain some a-priori estimates.
\begin{lemma}\label{bkm1}
Suppose that $\omega_0,\rho_0 \in K_n,$ $n \geq 1,$ and that $\omega(x,t)$ and $\Phi(x,t)$ satisfy \eqref{vortx}, \eqref{phieq} for all $x \in D,$ $0 \leq t \leq T.$
Assume that \begin{equation}\label{condreg} {\rm sup}_{x \in D, \, 0 \leq t \leq T} \left| \Phi(\log(x_1x_2),t) \right| \leq B < \infty.\end{equation}
Then $\omega(x,t),\rho(x,t) \in C([0,T],K_n).$
\end{lemma}
\begin{proof}
Let $\delta(t) = {\rm min}(\delta_{\omega(\cdot, t)}, \delta_{\rho(\cdot,t)}).$
From \eqref{vortx}, \eqref{rhox} and \eqref{condreg} it follows that
\[ \|\omega(\cdot, t)\|_{L^\infty} \leq \|\omega_0\|_{L^\infty} + \delta(0)^{-1} e^{3B} T \|\rho_0\|_{L^\infty}, \]
for every $0 \leq t \leq T.$ Moreover, $\delta(t) \geq \delta(0) e^{-B/2}.$
It remains to estimate the derivatives of the solution. Let us consider the first order derivatives of $\omega(x,t),$ $\rho(x,t)$ in \eqref{rhox}, \eqref{vortx}.
Let us use the representation
\begin{equation}\label{phixcord} \Phi(\log(x_1x_2), t) = 2\int_0^t \Omega(x_1x_2,s)\,ds = 2\int_0^t \int_{y_1y_2 \geq x_1x_2}\frac{\omega(y,s)}{|y|^2}\,dyds \end{equation}
from \eqref{phieq}. Then the only expression which appears that is not already clearly controlled is
$\partial_\eta \Omega(\eta, t),$ where we are using the $x$ coordinate representation
\[ \Omega(\eta,t) = \int_{y_1y_2 \geq \eta}\frac{\omega(y,s)}{|y|^2}\,dy. \]
But we can estimate
\[ \partial_\eta \Omega(\eta, t) =  \lim_{h \rightarrow 0} \frac{1}{h} \int_{\delta(t)}^C dy_1 \int_{\frac{\eta}{y_1}}^{\frac{\eta+h}{y_1}} dy_2 \frac{\omega(y,t)}{|y|^2} =
\int_{\delta(t)}^C \frac{dy_1}{y_1} \frac{\omega(y_1, \frac{\eta}{y_1}, t)}{y_1^2 + \left(\frac{\eta}{y_1}\right)^2} \leq C \|\omega(\cdot, t)\|_{L^{\infty}} \delta(t)^{-2}. \]
Higher order derivatives can be estimated inductively up to the level of regularity for the initial data - second order derivative for $\omega$ will involve terms that are clearly bounded plus
second order derivatives of $\Omega(\eta, t)$ which can be controlled by using bounds on the first order derivatives of $\omega(x,t)$ and so on. Continuity of the derivatives of $\omega$ and $\rho$
in time follows from continuity in time for $\omega(x,t)$ and $\rho(x,t)$ as is clear from \eqref{rhox}, \eqref{vortx} and an inductive argument.
\end{proof}

Due to Lemma~\ref{bkm1}, to prove Theorem~\ref{local} it suffices to construct a solution with a bounded $\Phi$ and $\omega$ (and then to address uniqueness). For this purpose, it will be more
convenient for us to work in the $z$ coordinates. The key equations are clearly the vorticity equation \eqref{omsol} and the phase equation \eqref{phidef}; the equation for density effectively
decouples and can be easily solved once we have solved the other two.


Now we sketch the proof of the existence and uniqueness of solutions to \eqref{omsol}, \eqref{phieq}. \\
\begin{proof}[Proof of Theorem~\ref{local}]
1. Uniform bound on the iterates.
Set $\omega_0(z,t)=\omega_0(z),$ $\Phi_0(z_1,t)=0.$ Iteratively, define
\begin{eqnarray}\label{omit}
\omega_n(z,t) = \omega_0(z_1, z_2 - \Phi_{n-1}(z_1,t))+f_1(z_1,z_2-\Phi_{n-1}(z_1,t)) \int_0^t e^{\Phi_{n-1}(z_1,s)}\,ds; \\
\Omega_n(z_1,t) = \frac14 \int_{z_1}^\infty d\tilde{z}_1 \int_{\R} \frac{\omega_n(\tilde{z},s)}{\cosh \tilde{z}_2}\,d\tilde{z}_2;
\label{Omegait} \\ \label{phiit}
\Phi_n(z_1,t) = 2 \int_0^t \Omega_n(z_1,s)\,ds.
\end{eqnarray}
Let us set
\begin{eqnarray*}
\Gamma_n(t) = {\rm sup}_{z_1, m \leq n} |\Omega_m(z_1,t)| \equiv \frac12 {\rm sup}_{z_1, m \leq n} |\partial_t \Phi_m(z_1,t)|; \\
M_n(t) = {\rm sup}_{z, \,m \leq n}|\omega_m(z,t)|, \,\,\,\,L_n(t) = {\rm sup}_{z_1, \,m \leq n}|\Phi_m(z_1,t)|.
\end{eqnarray*}
Observe that
\begin{eqnarray}
L_n(t) \leq 2 \int_0^t \Gamma_n(s)\,ds \label{lcon} \\
\Gamma_n(t) \leq C (1+L_{n-1}(t))M_n(t). \label{gamcon}
\end{eqnarray}
Indeed, due to our assumptions on the initial data \eqref{zsupp} and the structure of the solution \eqref{omsol}, we have
\[ {\rm supp} \,\omega_n(z_1,z_2, t) \subset \{ z_2 \leq z_1 +C + L_{n-1}(t), \} \]
with $C=-2\log \delta(0)$ according to \eqref{zsupp} and so depending only on the initial data.
In general, throughout the paper $C$ will denote constants that may change line to line but can only depend on the initial data; sometimes
we will make this dependence explicit.
It follows that for every $z_1,$
\begin{equation}\label{Omest11} |\Omega_n(z_1,t)| \leq C \int_{\tilde{z}_1 \geq z_1} d\tilde{z}_1 \int_{-\infty}^{\tilde{z}_1+C + L_{n-1}(t)} \frac{M_n(t) dz_2}{\cosh z_2} \leq
CM_n(t) (1+ L_{n-1}(t)), \end{equation} with constant $C$ independent of $n.$
Since by definition $L_{n-1}(t) \leq L_n(t),$ the estimates \eqref{lcon}, \eqref{gamcon}, \eqref{Omest11} and Gronwall lemma together imply that
\begin{equation}\label{Lcontrol} L_n(t) \leq e^{C \int_0^t M_n(s)\,ds}-1. \end{equation}
Then \eqref{omsol} leads to
\begin{equation}\label{eqF} M_n(t) \leq C \left(1+ \int_0^t e^{C\int_0^s M_n(r)\,dr}\,ds \right), \,\,\,M_n(0)=\|\omega_0\|_{L^\infty}. \end{equation}
Therefore, $M_n(t) \leq \bar M_n (t),$ where $\bar M_n (t)$ satisfies equality instead of an inequality in \eqref{eqF}.
Clearly there exists $T>0$ such that $\bar M_n \leq C < \infty$ for every $t \in [0,T].$ Then \eqref{lcon}, \eqref{gamcon} imply that $L_n(t)$
is bounded on $[0,T]$ as well, with $T$ as well as the upper bound $C$ independent of $n.$

2. Convergence. We now show that on time interval $[0,T],$ the approximations $\omega_n,$ $\Phi_n$ converge uniformly over the compact sets in $\R^2$
to bounded functions $\omega,$ $\Phi$ which solve \eqref{omsol}, \eqref{phidef}.

Let us denote
\[ G_n(z_1,t) = {\rm sup}_{z_2, \,\tilde{z}_1 \geq z_1, \, m \leq n}|\omega_m(\tilde{z}_1, z_2, t) - \omega_{m-1}(\tilde{z}_1, z_2, t)|. \]
Observe that
\begin{eqnarray} |\Phi_n(z_1,t)- \Phi_{n-1}(z_1,t)| \leq \left| \int_0^t ds \int_{z_1}^{C} d\tilde{z}_1 \int \frac{\omega_n(\tilde{z}_1, z_2, s) -
\omega_{n-1}(\tilde{z}_1, z_2, s)}{\cosh z_2}\,dz_2 \right| \nonumber \\
\leq C (|z_1|+1) \int_0^t G_n (z_1, s)\,ds; \label{exist123}
\end{eqnarray}
the constants $C$ here depend only on the initial data.
On the other hand, for $t \in [0,T]$ we have
\begin{eqnarray}
|\omega_n(z_1,z_2,t)-\omega_{n-1}(z_1,z_2,t)| \leq \left(\|\nabla\omega_0\|_{L^\infty} +C\|\nabla f_1\|_{L^\infty}\right)|\Phi_{n-1}(z_1,t)-\Phi_{n-2}(z_1,t)| \nonumber \\
+C \int_0^t |\Phi_{n-1}(z_1,s)-\Phi_{n-2}(z_1,s)|\,ds. \label{exist124}
\end{eqnarray}
Here $C$ needs to be chosen large enough so that $(\|f_1\|_{L^\infty}+\|\nabla f_1\|_{L^\infty})e^{\|\Phi_m(\cdot,t)\|_{L^\infty}}\leq C$ for $m=n-1,n-2$ and
$t \in [0,T].$ By combining the bounds \eqref{exist123}, \eqref{exist124}, we find
\[  G_n(z_1,t) \leq C(1+|z_1|)\int_0^t G_{n-1}(z_1,s)\,ds \] 
with a constant $C$ that depends only on the initial data and $T.$
Iterating, we obtain \[ G_n(z_1,t) \leq \frac{C^n (1+|z_1|)^n t^n}{n!}. \]
This makes $\omega_n$ Cauchy on all plane for every $t \in [0,T];$ convergence is uniform on any set $z_1 \geq A,$ $0 \leq t \leq T.$
Such convergence as well as uniform $L^\infty$ bound on $\omega_n$ also implies the same type of convergence $\Phi_n(z_1,t) \rightarrow \Phi(z_1,t).$
Moreover, it is straightforward to check that $\omega,$ $\Phi$ satisfy \eqref{omsol}, \eqref{phidef}.

3. Uniqueness. The uniqueness of a bounded solution $\omega,$ $\Phi$ on $[0,T]$ follows very similarly to the convergence part by looking at the
difference of two solutions and using the upper bound on solutions as well as the resulting differential inequality. We omit the details.

By Lemma~\ref{bkm1}, the proof of Theorem~\ref{local} is now complete.
\end{proof}

Finally, let us state one more regularity criterion that is claimed in Theorem~\ref{bu}. This result is the direct analog of the well known Beale-Kato-Majda criterion for the 3D Euler equation \cite{BKM}.
\begin{proposition}\label{truebkm}
Let $\omega(x,t),$ $\rho(x,t)$ be $C([0,T],K_n)$ solution of the system \eqref{vortmod}, \eqref{denmod}, \eqref{bsmod}.
If $T$ is the largest time of existence of such solution, then we must have
\[ \lim_{t \rightarrow T} \int_0^t \|\omega(\cdot, t)\|_{L^\infty}\,dt = \infty. \]
\end{proposition}
\begin{proof}
Define
\[ L(t) = {\rm sup}_{z_1}|\Phi(z_1,t)|. \]
An argument parallel to one leading to \eqref{Lcontrol} then gives global in time bound
\begin{equation}\label{Lexpcon1} L(t) \leq C e^{C\int_0^t\|\omega(\cdot,s)\|_{L^\infty}\,ds}, \end{equation}
where the constant $C$ only depends on the initial data.
Now global regularity follows from Lemma~\ref{bkm1}.
\end{proof}

The last issue we would like to discuss in this section is to come back to the different choice of the kernel in the definition of $\Omega$ in \eqref{bsmod}.
In the introduction, we mentioned that taking this kernel to be $y_1y_2/|y|^4,$ which follows \eqref{ML} more closely, does not result in an essential change of the analysis of the system.
Indeed, the analog of $z$ representation of $\Omega$ in \eqref{tildeOm} becomes
\[ \tilde{\Omega}(z_1,t) = \frac18 \int_{z_1}^\infty d\tilde{z}_1 \int_{\R} \frac{\tilde{\omega}(\tilde{z},t)}{(\cosh \tilde{z}_2)^2}\,d\tilde{z}_2, \] 
and analysis in this section as well as below can proceed along the same path and with identical conclusion.

\section{The analog of the 2D Euler equation}\label{2deul}

The goal of this section is to provide more intuition on properties of the ``hyperbolic" Biot-Savart law \eqref{bsmod}. For this purpose, we will consider
``hyperbolic" 2D Euler equation, which is obtained by taking $\rho \equiv 0$ in \eqref{vortmod}:
\begin{eqnarray}
\label{vort2de}
\partial_t \omega + (u \cdot \nabla) \omega = 0, \,\,\, \omega(x,0)=\omega_0(x), \\
u = (-x_1 \Omega(x_1x_2,t), x_2 \Omega(x_1x_2,t)), \,\,\, \Omega(x_1x_2,t) = \int_{y_1y_2 \geq x_1x_2}\frac{1}{|y|^2}\omega(y,t)\,dy. \label{bs2de}
\end{eqnarray}
 We will see that, similarly to the 2D Euler equation,
its hyperbolic analog is globally regular. However, there is one interesting difference - the hyperbolic version of 2D Euler is in some sense
more regular than the real 2D Euler equation. Namely, the rate of growth of the derivatives of solutions of the hyperbolic 2D Euler equation can only be
exponential in time. For the sake of simplicity, we will restrict ourselves to the initial data that is positive on $D$ - note that the double
exponential growth examples of \cite{KS} involve exactly this class of the initial data.

\begin{theorem}\label{hypexp}
Suppose $\omega_0 \in K_n,$ $n \geq 1.$ Then the system \eqref{vort2de}, \eqref{bs2de} set in $D$ has a unique global solution $\omega(x,t) \in K_n.$
Moreover, assume $\omega_0 \geq 0$ in $D.$ Then
\begin{equation}\label{expupper}
\|\omega(\cdot,t)\|_{K_n} \leq C e^{Ct},
\end{equation}
where the constant $C$ only depends on the initial data.

Moreover, there exist initial data $\omega \in K_n,$ $n \geq 1,$ for which the exponential in time growth of derivatives  (including the first order ones)
of the corresponding solution is realized.
\end{theorem}

Global regularity of the solution follows immediately from Proposition~\ref{truebkm}, once we observe that
$\|\omega(\cdot,t)\|_{L^\infty} = \|\omega_0\|_{L^\infty}$ while the solution is still regular.
Moreover, from the explicit formula for solution \eqref{vortx}, it follows that
higher order derivatives of $\omega(x,t)$ satisfy double exponential upper bound in time:
\begin{lemma}\label{doublexp}
Suppose $\omega_0 \in K_n,$ $n \geq 1,$ and $\omega(x,t)$ is the corresponding global solution of \eqref{vort2de}, \eqref{bs2de} in $D.$
Then the higher order derivatives of $\omega(x,t)$ satisfy
\begin{equation}\label{upperdoubleexp}
\|D^l \omega (x,t)\|_{L^\infty} \leq Ce^{C\Phi(\log (x_1x_2),t)} \leq Ce^{Ce^{Ct}}
\end{equation}
for every $1 \leq l \leq n,$ where the constant $C$ depends only on $\omega_0$ and $l.$
\end{lemma}
\begin{proof}
Recall our definition $\delta_f = {\rm min}_{x \in {\rm supp}(f)} x_1,$ and let us use a short cut $\delta(t) = \delta_{\omega(\cdot, t)}.$
Without loss of generality, to simplify the computations, we will assume $\delta(0)<1.$
When $\rho=0,$ \eqref{vortx} transforms into
\[ \omega(x,t) =  \omega_0\left(x_1 e^{\frac12 \Phi(\log(x_1x_2),t)},
x_2 e^{-\frac12 \Phi(\log(x_1x_2),t)}\right). \]
Therefore,
\begin{equation}\label{deltaest} \delta(t) \leq \delta(0)e^{-\frac12 \Phi(-\infty,t)}. \end{equation}
On the other hand, \eqref{Lexpcon1} and conservation of the $L^\infty$ norm of vorticity imply that
\begin{equation}\label{phidest} \Phi(-\infty,t) \leq C e^{Ct}. \end{equation}
The bounds on the derivatives of $\omega(x,t)$ now follow by direct differentiation and estimates similar to the ones described in the proof of
Lemma~\ref{bkm1}.
\end{proof}
Lemma~\ref{doublexp} is in close parallel to the corresponding result for the
classical 2D Euler equation. Our next goal is to prove a sharper upper bound which is just exponential in time (it will not be hard to see that it is in fact optimal).

Consider first the degenerate case where $\omega_0 \equiv 0$ on the boundary of the quadrant $\partial D.$ It will be easier for us to work
in the $x$ coordinate representation.

Recall the representation \eqref{phixcord}, and note that since $\omega(x,t) \geq 0$ in $D$, we have
$\Omega(0,t) = {\rm max}_{\eta \geq 0} \Omega (\eta,t).$
Since $D$ is closed and $\omega_0 \in C^1(D),$ we have $\omega_0(x_1,x_2) \leq C {\rm min}(1,x_2).$  Therefore
\begin{eqnarray*} \Omega(0,t) = \int_D \frac{\omega_0\left(x_1 e^{\frac12 \Phi(\log(x_1x_2),t)}, 
x_2 e^{-\frac12 \Phi(\log(x_1x_2),t)}\right)}{x_1^2 + x_2^2}\,dx \leq \\
\int_D \frac{C{\rm min}(1,x_2)}{x_1^2+x_2^2}\,dx \leq C \int_0^1 dx_2 \int_0^C \frac{x_2}{x_1^2 + x_2^2}\,dx_1
+ C \int_1^\infty dx_2 \int_0^{C} \frac{1}{x_1^2+x_2^2}\,dx_1 \leq C,
\end{eqnarray*}
where the constant depends only on the initial data. By the first inequality in \eqref{upperdoubleexp} and by \eqref{phixcord}, the bound
\eqref{expupper} follows.

Let us now consider the case where $\omega_0(x)$ does not identically vanish on $\partial D.$ As is clear from the preceding paragraph,
to prove Theorem~\ref{hypexp}, it suffices to show global uniform bound $\Omega(0,t) \leq A < \infty$ for all $t.$ This is exactly what we will
do. Note that this bound does not follow from the global $L^\infty$ control of $\omega(x,t);$ it is easy to see that the integral defining
$\Omega(0,t)$ diverges as a logarithm when the support of $\omega(x,t)$ approaches the origin. The uniform bound on $\Omega(0,t)$ will be
a consequence of the dynamical properties of the model.
The proof is not straightforward, and some of the auxiliary results that we develop here will also be useful for the proof of finite time
blow up in the next section. It will be convenient for us to work in the $z$ coordinates. Note that when $\rho \equiv 0,$
\eqref{omsol} transforms into
\begin{equation}\label{omsol2de}
\omega(z,t) = \omega_0\left(z_1, z_2 - \Phi(z_1,t)\right).
\end{equation}
Without loss of generality, we will assume in this section that $\|\omega_0\|_{L^\infty}=1$. The general case can be reduced to it by a simple change of time variable.

Let us define \[ Z = {\rm sup} \{ z_1 \left| \,\,\exists z_2: (z_1,z_2) \in {\rm supp} \omega_0. \right. \} \]
\begin{lemma}\label{suppinf}
Suppose that $0 \leq \omega_0 \in K_n,$ $n \geq 1.$ Then
for every $z_1 < Z,$ we have
\[ \Phi(z_1,t) \stackrel{t \rightarrow \infty}{\longrightarrow} \infty. \]
\end{lemma}
\begin{proof}
Suppose not: there exists $\bar z_1 <Z$ such that $\Phi(\bar z_1,t) \leq A < \infty$ for every $t.$
Then the same is true for every $z_1$ satisfying $\bar z_1 \leq z_1 \leq Z,$ since $\Phi(z_1,t) = \int_0^t \Omega(z_1, s)\,ds,$
and $\Omega(z_1,s)$ is monotone decaying in $z_1$ due to positivity of $\omega$. But then for every such $z_1$ and for all $t$ we have
\[ \int_{\R} \frac{\omega(z_1,z_2, t)}{\cosh z_2} \,dz_2 = \int_{\R} \frac{\omega_0(z_1-\Phi(z_1,t),z_2)}{\cosh z_2} \,dz_2 \geq C(\omega_0)e^{-|z_1|-A}>0. \]
This implies that $\Omega(\bar z_1, t) \geq C(\omega_0)\int_{\bar z_1}^Z e^{-|z_1|-A}\,dz_1 >0$ for every $t.$ But this bound contradicts our assumption on $\Phi(\bar z_1,t).$
\end{proof}

\begin{lemma}\label{constdrive}
Suppose that $\omega_0 \in K_n,$ $n \geq 1,$ $\omega_0 \geq 0$ and $\omega_0$ does not identically vanish on $\partial D.$
Then there exists $Z_1$ such that for every $z_1 \leq Z_1,$ we have
\[ \int_{\R} \omega_0(z_1,z_2)\,dz_2 \geq c(\omega_0) >0.  \]
\end{lemma}
\begin{proof}
Direct calculation shows that
\[ \int_{\R} \omega_0(z_1,z_2)\,dz_2 = 2\int_0^\infty \omega_0\left(x_1, \frac{e^{z_1}}{x_1}\right) \,\frac{dx_1}{x_1}, \]
where we switched to $\omega_0$ in the $x$ coordinates in the second integral. The integral on the right hand side is bounded
from below uniformly by a positive constant for all $z_1$ sufficiently small, due to the assumption that $\omega_0 \in C^1$ and
does not identically vanish on $\partial D.$
\end{proof}

Let us denote $H(z_1,t) = z_1 + \Phi(z_1,t).$
Due to \eqref{zsupp} and \eqref{omsol2de}, we see that
\[ {\rm supp} \,\omega(z,t) \subset \{ (z_1,z_2): \,\, H(z_1,t)-K \leq z_2 \leq H(z_1,t)+K \}, \] 
so that $H$ describes how close, for a fixed $z_1$ and $t$, is the support of the solution from the maximum of the weight $1/\cosh z_2.$

Define the forward front, $F_1(t),$ by
\begin{equation}\label{forfro}
F_1(t) = {\rm min} \{ z_1 \left| \,\,\, H(z_1, t) =-B  \right. \},
\end{equation}
where $B \geq 1$ is a sufficiently large constant that will be chosen below.
Note that due to our assumption on the initial data
\eqref{zsupp}, and \eqref{phidest}, we have that $H(z_1,t) \rightarrow -\infty$ as $z_1 \rightarrow -\infty$ for all $t.$
Then due to Lemma~\ref{suppinf}, $F_1(t)$ is well defined for all times $t \geq t_0.$
Without loss of generality, we can choose $t_0$ large enough so that $H(Z_1,t) \geq 0$ (and so in particular $F_1(t) \leq Z_1$)
for all $t \geq t_0$.
The time $t_0$ depends only on $\omega_0,$ and will be fixed throughout the argument of this section.

\begin{lemma}\label{Omlb}
For every $z_1 \leq Z_1$ such that $H(z_1,t) \leq -B,$ for every $t \geq t_0,$ we have $\Omega(z_1,t) \geq \gamma >0,$ where the
constant $\gamma$ only depends on $\omega_0.$
\end{lemma}
\begin{proof}
Observe that
\begin{equation}\label{partzD}
\partial_{z_1} H(z_1,t) = 1 + 2 \int_0^t \partial_{z_1}\Omega(z_1,s)\,ds \leq 1,
\end{equation}
since $\partial_{z_1}\Omega(z_1,t) \leq 0$ for all $z_1$ and $t.$ Let us denote $S$ the set of all $\tilde{z}_1 \in [z_1, Z_1]$
such that $|H(\tilde{z}_1,t)| \leq 1.$ Due to $H(z_1,t) \leq -B \leq -1,$ $H(Z_1,t) \geq 0,$ and \eqref{partzD}, it is
straightforward to see that $|S| \geq 1.$ Then, by \eqref{omsol2de},
\[ \Omega(z_1,t) \geq \frac14 \int_S d\tilde{z}_1 \int_{\R} dz_2 \frac{\omega_0(\tilde{z}_1, z_2 + \tilde{z}_1 -  H(\tilde{z}_1,t))}{\cosh z_2} \geq
\frac14 |S| c(\omega_0) e^{-K-1} \equiv \gamma, \]
where $c(\omega_0)$ is the constant from Lemma~\ref{constdrive}, and $K$ is the constant from \eqref{zsupp}.
Note that $\gamma$ is independent of $B.$
\end{proof}

The next proposition describes the structure of $H(z_1,t)$ for $z_1 \leq F_1(t).$
\begin{proposition}\label{frontstr}
For every $z_1 \leq F_1(t),$ we have
$\partial_{z_1}H(z_1,t) \geq 1 - (\gamma^{-1}+t_0)e^{K-B}.$ In particular, if we choose $B$ large enough so that $B \geq K,$ $B \geq 1,$ and $\epsilon \equiv (\gamma^{-1}+t_0)e^{K-B} < 1,$
then
\begin{equation}\label{pastfront}
1 \geq \partial_{z_1} H(z_1,t) \geq 1 -\epsilon >0
\end{equation}
for all $t \geq t_0,$ $z_1 \leq F_1(t).$
\end{proposition}
\begin{proof}
Observe that
\[ \partial_{z_1}H(z_1,t) = 1 - \frac12 \int_0^t ds \int_{\R} dz_2 \frac{\omega_0(z_1, z_2+z_1-H(z_1,s))}{\cosh z_2}, \]
so we need to estimate the integral on the right hand side. Since $z_1 \leq F_1(t),$ we have $H(z_1,s) \leq -B$
for all $t_0 \leq s \leq t.$ Therefore, by Lemma~\ref{Omlb}, we have $\Omega(z_1,s) \geq \gamma$ for every $t_0 \leq s \leq t.$
It follows that $H(z_1,s) \leq -B - 2\gamma(t-s)$ for all $t_0 \leq s \leq t.$ Using \eqref{zsupp} and $\|\omega_0\|_{L^\infty}=1,$ we can estimate
\begin{eqnarray*}
 \int\limits_0^t ds \int\limits_{\R} dz_2 \frac{\omega_0(z_1, z_2+z_1-H(z_1,s))}{\cosh z_2} \leq
\int\limits_0^{t_0} ds \int\limits_{\R} dz_2 \frac{\omega_0(z_1, z_2+z_1-H(z_1,s))}{\cosh z_2} + \\ \int\limits_{t_0}^t ds \int\limits_{\R} dz_2 \frac{\omega_0(z_1, z_2+z_1-H(z_1,s))}{\cosh z_2} \leq
2t_0 e^{K-B} + \int\limits_{t_0}^t ds \int\limits_{-\infty}^{K-B-2\gamma(t-s)} 2e^{z_2} \,dz_2 \leq (2t_0 + \gamma^{-1})e^{K-B}.
\end{eqnarray*}
\end{proof}
For the rest of this section, we will choose $B$ so that $\epsilon \equiv (\gamma^{-1}+2t_0)e^{K-B} \leq 0.1$.

Observe that by Proposition~\ref{frontstr}, with our choice of $B,$ the function $H(z_1,t)$ is strictly increasing in $z_1$ in the $z_1 \leq F_1(t)$ region.

Note also that Proposition~\ref{frontstr} implies that $F_1(t)$ is continuous in time. Indeed, a jump in $F_1(t)$ would only be possible if $D(z_1,t)$
were not strictly monotone for $z_1 \leq F_1(t).$

In fact, the proof of Proposition~\ref{frontstr} yields the following stronger statement.
\begin{corollary}\label{minB}
Suppose that $H(z_1,t) \leq -B$ for some $t \geq t_0.$ Then $\partial_{z_1} H(z_1, t) \geq 1- \epsilon.$ As a consequence, the function $H(z_1,t)$
is one-to-one in pre-image of $(-\infty,-B],$ and this pre-image equals $(-\infty, F_1(t)].$ In particular, $H(z_1,t) > -B$ for every $z_1 > F_1(t).$
\end{corollary}
\begin{proof}
The proof of the bound on $\partial_{z_1}H(z_1,t)$ is identical to that in the proof of Proposition~\ref{frontstr}. The rest of Corollary~\ref{minB}
follows immediately.
\end{proof}

One further consequence of Proposition~\ref{frontstr} is that to control $\Omega(-\infty,t),$ it suffices to estimate $\Omega(F_1(t),t).$
\begin{corollary}\label{Ominfest}
For every $t \geq t_0,$ we have
\[ \Omega(-\infty,t) \leq \Omega(F_1(t),t) + \frac{1}{1-\epsilon} e^{K-B}. \]
\end{corollary}
\begin{proof}
By Proposition~\ref{frontstr}, we have
\begin{eqnarray*}
\Omega(-\infty,t) - \Omega(F_1(t),t) = \frac14 \int\limits_{-\infty}^{F_1(t)} dz_1 \int\limits_{\R} dz_2 \frac{\omega_0(z_1, z_2+z_1-H(z_1,s))}{\cosh z_2} \leq \\
\frac14 \int\limits_{-\infty}^{F_1(t)} dz_1 \int\limits_{-\infty}^{K-B-|z_1-F_1(t)|(1-\epsilon)} e^{z_2} \,dz_2 \leq \frac14 e^{K-B} \int\limits_{-\infty}^{F_1(t)}
e^{-(1-\epsilon)|z_1-F_1(t)|}\,dz_1 \leq \frac{1}{4(1-\epsilon)} e^{K-B}.
\end{eqnarray*}
\end{proof}


Now we are ready to state a key proposition from which Theorem~\ref{hypexp} will follow.

\begin{proposition}\label{key2deprop}
Set $A_0 = 10000(B+K)^2 e^{2(B+K)}.$ Let $t_1$ be any time such that $\Omega(F_1(t_1), t_1)=A_0.$ Then there exists $\delta t = \delta t (B,K) \leq 1$ such that
$\Omega(F_1(t_1+\delta t), t_1+\delta t) \leq A_0/2,$ and $\Omega(-\infty,t) \leq 2A_0 e^{B+K}$ for all $t_1 \leq t \leq t_1 +\delta t.$
\end{proposition}
From Proposition~\ref{key2deprop} it follows, of course, that $\Omega(-\infty, t)$ is globally bounded, thus proving the first part of Theorem~\ref{hypexp}.
\begin{proof}
Set $\delta t = 100(B+K)e^{B+K}A_0^{-1}.$ Define $R(t_1)$ by the condition
\[ \frac{1}{\delta t} \int_{t_1}^{t_1+\delta t} \Omega(R(t_1),s)\,ds = \frac{1}{10}A_0 e^{-(B+K)}. \]
Note that if $R(t_1)$ exists, then it is unique due to monotonicity of the left hand side in the first argument of $\Omega$ (the only possible exception is if
$\omega_0$ vanishes for a range of $z_1$ and $R(t_1)$ fits exactly there; this exception is trivial as this range of $z_1$ can be simply collpased into a single point
without affecting anything). Let us consider two cases.

1. Suppose that $R(t_1)$ exists and $R(t_1) \geq F(t_1).$ In this case, we claim that
\begin{equation}\label{Rdeltab} \Omega(R(t_1), t_1 + \delta t) \leq \frac{A_0}{10}. \end{equation}
Indeed, by mean value theorem, we can find $t_2 \in [t_1, t_1+\delta t]$
such that \begin{equation}\label{Omt2b} \Omega(R(t_1),t_2) = \frac{1}{10}A_0 e^{-(B+K)}. \end{equation}
Note that for every $z_1 \geq R(t_1) \geq F(t_1),$ by Corollary~\ref{minB} we have $H(z_1,t) \geq -B$ for $t \geq t_1.$ The contribution of such $z_1$ to the integral providing the value of $\Omega(R(t_1),t_2)$
is equal to \[ \int_{\R} \frac{\omega_0(z_1, z_2+z_1 - H(z_1,t_2))}{\cosh z_2}\,dz_2. \]
Due to \eqref{zsupp} and the inequality $H(z_1,t_2) \geq -B,$ this integral can increase by a factor of at most $e^{B+K}$ over the remaining times:
\[ {\rm sup}_{s \geq t_2}  \int_{\R} \frac{\omega_0(z_1, z_2+z_1 - H(z_1,s))}{\cosh z_2}\,dz_2 \leq e^{B+K}  \int_{\R} \frac{\omega_0(z_1, z_2+z_1 - H(z_1,t_2))}{\cosh z_2}\,dz_2. \]
This and \eqref{Omt2b} imply \eqref{Rdeltab}.

Observe that an identical argument shows that
\begin{equation}\label{OmFb}
\Omega(F_1(t_1),t) \leq A_0 e^{B+K}
\end{equation}
for all $t \geq t_1,$ something that we will need later.

Let us now show that by time $t_1+\delta t,$ the points $z_1$ satisfying $F(t_1) \leq z_1 \leq R(t_1)$ do not contribute much to $\Omega.$
By definition of $R(t_1),$ we have
\[ \int_{t_1}^{t_1+\delta t} \Omega(R(t_1),s)\,ds = 10(B+K). \]
On the other hand, $H(z_1,t_1) \geq -B.$ Then
\[ H(z_1, t_1+\delta t) \geq -B + 20(B+K). \]
Therefore,
\[ \int_{\R} \frac{\omega_0(z_1, z_2 +z_1 - H(z_1,t_1+\delta t))}{\cosh z_2}\,dz_2 \leq e^{-10(B+K)} \int_{\R}
\frac{\omega_0(z_1, z_2 +z_1 - H(z_1,t_1))}{\cosh z_2}\,dz_2. \]
Hence,
\begin{equation}\label{intermed}
\Omega(F_1(t_1),t_1+\delta t) - \Omega(R(t_1),t_1+\delta t) \leq e^{-10(B+K)} \Omega(F_1(t_1),t_1) = e^{-10(B+K)} A_0.
\end{equation}

It remains to consider the contribution of $z_1 \leq F(t_1)$. Notice that due to \eqref{OmFb}, we have
\begin{equation}\label{omfex} \int_{t_1}^{t_1+\delta t} \Omega(F_1(t_1),s)\,ds  \leq \delta t A_0 e^{B+K} \leq 100 (B+K) e^{2(B+K)}. \end{equation}
Set $Y = F_1(t_1)- 1000(B+K) e^{2(B+K)}.$ We claim that $Y \leq F_1(t_1+\delta t),$ that is, for every $z_1 \leq Y$
we have
$H(z_1,t_1+\delta t) < -B.$
To show this, note first that by Proposition~\ref{frontstr}, $H(z_1,t_1) \leq -B -(1-\epsilon)(F_1(t)-z_1).$
Next, using \eqref{omfex} and $\|\omega(\cdot, t)\|_{L^\infty} =1$, we estimate
\begin{eqnarray*}
\int_{t_1}^{t_1+\delta t} \Omega(z_1,s)\,ds \leq \int_{t_1}^{t_1+\delta t} \Omega(F_1(t_1),s)\,ds + \int_{t_1}^{t_1+\delta t} \int_{z_1}^{F_1(t_1)} d \tilde{z}_1
\int_{\R} \frac{\omega(z,s)}{\cosh z_2}\,dz_2 ds \leq \\ 100 (B+K) e^{2(B+K)} + 2 (F_1(t)-z_1) \delta t \leq 100 (B+K) e^{2(B+K)} + \frac{1}{50} e^{-B-K} (F(t_1)-z_1).
\end{eqnarray*}
To ensure that $H(z_1,t_1+\delta t) < -B$ holds for $z_1 \leq Y,$ it suffices to verify that
\[ (1-\epsilon)(F_1(t_1)-Y) \geq 200 (B+K) e^{2(B+K)} + \frac{1}{25} e^{-B-K} (F(t_1)-Y). \]
This clearly holds true by our choice of $B$ in Proposition~\ref{frontstr}.
Now since $F_1(t_1+\delta t) \geq Y,$ the contribution of all $z_1 \leq F_1(t_1)$ to $\Omega(-\infty,t)$ for times $t_1 \leq t \leq t_1 +\delta t$ cannot exceed
\begin{equation}\label{restofz}
\Omega(-\infty,t)-\Omega(F_1(t_1),t) \leq 2000(B+K)e^{2(B+K)} + \frac{1}{1-\epsilon}e^{K-B}
 \end{equation}
by direct estimate and Corollary~\ref{Ominfest}.

Combining estimates \eqref{Rdeltab}, \eqref{intermed} and \eqref{restofz} together, we find that
\[ \Omega(F_1(t_1+\delta t), t_1+\delta t) \leq \Omega(-\infty, t_1+\delta t) \leq \frac{A_0}{10} + A_0 e^{-10(B+K)}+2000(B+K)e^{2(B+K)}+ 2e^{K-B} < \frac{A_0}{2} \]
by definition of $A_0$ and our choice of $B.$

In addition, by \eqref{OmFb} and \eqref{restofz}, for every $t_1 \leq t \leq t_1+\delta t$ we have
\[ \Omega(-\infty, t) \leq A_0 e^{B+K} + 2000(B+K)e^{2(B+K)}+ 2e^{K-B} \leq 2A_0 e^{B+K}. \]

2. Suppose now that $R(t_1) < F(t_1)$ or does not exist. This case is easier. We now have
\[ \frac{1}{\delta t} \int_{t_1}^{t_1+\delta t} \Omega(F_1(t_1), s) \,ds \leq \frac{A_0}{10} e^{-(B+K)}, \]
and $\Omega(F_1(t_1), t_1+\delta t) \leq A_0/10$ by the same argument as the bound for $\Omega(R(t_1), t_1+\delta t)$
in the first case. We also have \[ \Omega(F_1(t_1), t) \leq A_0 e^{B+K} \] for all $t_1 \leq t \leq t_1+\delta t.$
Thus the range $z_1 \geq F_1(t_1)$ is controlled. the estimate for $z_1 \leq F_1(t_1)$ proceeds similarly to the first case,
but now we have a better bound
\[ \int_{t_1}^{t_1+\delta t} \Omega(F_1(t_1), s) \,ds \leq \delta t \frac{A_0}{10} e^{-(B+K)} \leq 10 (B+K). \]
Similarly to the first case, we can show that $F_1(t_1+\delta t) \geq F_1(t_1) - 20 (B+K).$
The rest of the argument is parallel to the first case and in fact yields better bounds.
\end{proof}
Finally, we note that \eqref{deltaest} and Lemma~\ref{Omlb} can be used in a straightforward way to show existence of solutions to the hyperbolic 2D Euler equation with exponential
growth of derivatives. In fact, such behavior is generic if $\omega_0$ is non-negative in $D$ and does not identically vanish on the boundary. This observation completes the proof
of Theorem~\ref{hypexp}.

\section{Finite time blow up}\label{blowup}

We now come back to the full system \eqref{omsol}, \eqref{phidef} and prove Theorem~\ref{bu}.
For the sake of simplicity, we will assume that
$\omega_0 \equiv 0,$ while $\rho_0 \in K_n,$ $n \geq 1,$ $\rho_0 \geq 0,$ and $\rho_0$ does not identically vanish on $\partial D.$
We will assume that the solution stays globally regular, and obtain a contradiction.
The hyperbolic Boussinesq system in the integral form can be reduced to the single equation
\begin{equation}\label{phieff}
\Phi(z_1,t) =  \frac12 \int_0^t \int_{z_1}^\infty  \int_{\R} \frac{f_1\left(\tilde{z}_1, z_2 - \Phi(\tilde{z}_1,s)\right) \int_0^s e^{\frac12 \Phi(\tilde{z}_1,r)}\,dr}{\cosh z_2}\,dz_2 d\tilde{z}_1 ds,
\end{equation}
with $f_1(z_1,z_2) = e^{\frac{z_2-z_1}{2}} \rho_0(z_1,z_2).$

Define \[ F_2(t) = {\rm max} \{ z_1 \left| H(z_1,t)=0 \right. \}. \]
Clearly, an analog of Lemma~\ref{suppinf} holds for the full system by an argument completely parallel to the 2D hyperbolic Euler case.
It follows that $F_2(t)$ is well-defined for all $t$ larger than $t_0$ which only depends on $\rho_0.$ Moreover, $F_2(t)$ is monotone decreasing, perhaps with jumps,
and tends to $-\infty$ as time advances. Let us define $Z_2$ in the same fashion as $Z_1$ in Lemma~\ref{constdrive}, but for $f_1$ instead of $\omega_0:$
$Z_2$ is the maximal value such that for every $z_1 \leq Z_2,$ we have
\begin{equation}\label{f1bound1} \int_{\R} f_1(z_1,z_2)\,dz_2 \geq c(\rho_0) >0.  \end{equation}
Let us choose $t_2>0$ so that $F_2(t_2) = {\rm min}(-10,Z_2)$ or $t_2=0$ if $F_2(0) \leq {\rm min}(-10,Z_2).$

\begin{lemma}\label{driveb}
For every $t \geq t_2,$ for every $z_1 \in [F_2(t), F_2(t)+1],$ for every $s$ such that $t_2 \leq s \leq t,$ we have
\[ \Phi(z_1,s) \geq \frac12 |F_2(s)|. \]
\end{lemma}
\begin{proof}
Let $s \geq t_2,$ and $z_1 \in [F_2(s), F_2(s)+1].$ Then due to the definition of $F_2,$ we have $H(z_1,s) \geq 0.$
This implies that \[ \Phi(z_1,s) \geq -z_1 \geq |F_2(s)| - 1 \geq \frac{|F_2(s)|}{2}. \]
Now if $z_1 \in [F_2(t), F_2(t)+1],$ then $\Phi(z_1,s)$ is even larger since $F_2$ is monotone decreasing in time and $\Phi$ is monotone
decreasing in $z_1.$
\end{proof}

\begin{proof}[Proof of Theorem~\ref{bu}]
Consider the identity $H(F_2(t),t)=0,$ $t \geq t_2.$ Since $F_2$ is monotone, it is differentiable for a.e. $t.$ thus for a.e. $t$ we have
\begin{equation}\label{keyeqdiff} 0 = \frac{d}{dt}H(F_2(t),t) = \partial_{z_1} H(F_2(t),t) F_2'(t) + \partial_t H(F_2(t), t). \end{equation}
Observe that by Lemma~\ref{driveb}, we have
\begin{eqnarray*}
\partial_t H(F_2(t),t) = \frac12 \int_{F_2(t)}^\infty \int_{\R} \frac{f_1\left(z_1, z_2 - \Phi(z_1,t)\right) \int_0^t e^{\frac12 \Phi(z_1,s)}\,ds}{\cosh z_2}\,dz_2 dz_1 \geq \\
 \frac12 \int_{t_2}^t e^{\frac14 |F_2(s)|}\,ds \int_{F_2(t)}^{F_2(t)+1} \int_{\R} \frac{f_1\left(z_1, z_2 +z_1 - H(z_1,t)\right)}{\cosh z_2} \,dz_2 dz_1.
\end{eqnarray*}
Due to the definition of $F_2$ and the bound $\partial_{z_1}H(z_1,t) \leq 1,$ we have that for $z_1 \in [F_2(t),F_2(t)+1],$ the inequality $0 \leq H(z_1,t) \leq 1$ holds.
Then, using \eqref{zsupp} and \eqref{f1bound1}, we can estimate
\[  \int_{F_2(t)}^{F_2(t)+1} \int_{\R} \frac{f_1\left(z_1, z_2 +z_1 - H(z_1,t)\right)}{\cosh z_2} \,dz_2 dz_1 \geq C(\rho_0) \int_{F_2(t)}^{F_2(t)+1} \frac{1}{\cosh H(z_1,t)} \,dz_1
\geq C_1(\rho_0)>0. \]
Therefore, we arrive at the bound
\begin{equation}\label{parttH}
\partial_t H(F_2(t),t) \geq C  \int_{t_2}^t e^{\frac14 |F_2(s)|}\,ds,
\end{equation}
where the constant $C$ depends only on $\rho_0.$ On the other hand, it follows from our usual estimate and the definition of $F_2$ that $1 \geq \partial_{z_1} H(F_2(t),t) \geq 0.$
Combining this bound with \eqref{keyeqdiff} and \eqref{parttH}, we obtain
\[ F_2'(t) \leq -C \int_{t_2}^t e^{\frac14 |F_2(s)|}\,ds, \]
for a.e. $t.$ Applying this differential inequality, it is straightforward to show that $F_2(t)$ tends to $-\infty$ in finite time. Therefore, $\Phi(-\infty, t)$ becomes infinite
in finite time. Explicit formulas for the solution show that this means finite time blow up; by Proposition~\ref{truebkm}, this also implies that
\[ \lim_{t \rightarrow T} \int_0^t \|\omega(\cdot, t)\|_{L^\infty}\,dt = \infty \]
for the blow up time $T< \infty.$
\end{proof}

\vspace{0.5cm}

\textbf{Acknowledgment.}
The authors acknowledge partial support of the NSF-DMS grant 1412023.


\begin{thebibliography}{99}

\bibitem{BKM}
J. T. Beale, T. Kato, and A. Majda,  \it Remarks on the breakdown of smooth solutions for the 3-D Euler equations, \rm
Commun. Math. Phys., {\bf 94}, pp. 61--66, 1984

\bibitem{CCW} D. Chae, P. Constantin and J. Wu, \it An incompressible 2D didactic model with singularity and explicit
solutions of the 2D Boussinesq equations, \rm Journal of Mathematical Fluid Mechanics, {\bf 16} (2014), 473--480


\bibitem{CHKLSY}
K.~Choi, T.~~Y. Hou, A.~ Kiselev, G.~Luo, V.~Sverak, and
  Y.~Yao, \it On the finite-time blowup of a 1d model for the 3d axisymmetric Euler
  equations,
\rm arXiv preprint arXiv:1407.4776, to appear at Commun. Pure Appl. Math.



\bibitem{CKY}
K. Choi, A. Kiselev, and Y. Yao, \it
Finite time blow up for a 1D model of 2D Boussinesq system, \rm
Comm. Math. Phys., {\bf 3} (2015), 1667--1679

\bibitem{DKX} T. Do, A. Kiselev and X. Xu, \it Stability of blow up for a 1D model of axisymmetric 3D Euler equation,
\rm preprint	arXiv:1604.07118




\bibitem{HR} V. Hoang and M. Radosz, \it in preparation \rm

\bibitem{HORY} V.~Hoang, B.~Orcan, M.~Radosz and H.~Yang, \it Blowup with vorticity control for a 2D model of Boussinesq equations,
\rm preprint

\bibitem{HouLiu}
T.~Y.~Hou and P.~Liu, \it Self-similar singularity of a 1D model for the 3D axisymmetric Euler equations, \rm
Research in Mathematical Sciences, {\bf 2:5} (2015), 1--26

\bibitem{KS} A.~Kiselev and V.~Sverak, \it Small scale creation for solutions of the incompressible two dimensional Euler equation, \rm
Annals of Math. {\bf 180} (2014), 1205--1220

\bibitem{HouLuo1}
G.~Luo and T.~Y. Hou,
\it Towards the finite-time blowup of the 3d axisymmetric Euler
  equations: A numerical investigation, \rm
Multiscale Model. Simul., {\bf 12} (2014), 1722--1776

\bibitem{MB}
A.~J. Majda and A.~L. Bertozzi,
\it Vorticity and Incompressible Flow, \rm
Cambridge University Press, 2002

\bibitem{MP} C.~Marchioro and M.~Pulvirenti, \it Mathematical Theory of Incompressible Nonviscous Fluids, \rm
Applied Mathematical Sciences Series (Springer-Verlag, New York), {\bf 96}, 1994

\bibitem{TT} T.~Tao, \it Finite time blowup for Lagrangian modifications of the three-dimensional Euler equation, \rm
preprint arXiv:1606.08481v1

\bibitem{Wol} W. Wolibner, \it Un theor\`eme sur l’existence du mouvement plan d’un fluide parfait, homog\`ene, incompressible,
pendant un temps infiniment long (French), \rm Mat. Z., {\bf 37} (1933), 698--726

\bibitem{Yud} V.~I.~Yudovich, \it Elelven great problems of mathematical hydrodynamics, \rm Moscow Mathematical Journal,
{\bf 3} (2003), 711-–737


\end{thebibliography}

\end{document}